\documentclass[11pt]{article}

\usepackage[top=1.6in, bottom=1.7in, left=1.6in, right=1.6in]{geometry}

\usepackage{latexsym,amsmath}
\usepackage{algpseudocode}

\newtheorem{theorem}{Theorem}

\newtheorem{corollary}{Corollary}
\newtheorem{definition}{Definition}
\newtheorem{remark}{Remark}

\newtheorem{proposition}{Proposition}

\newenvironment{proof}{{\it Proof:\/}}{\hfill $\Box$\\ }
\newtheorem{heuristics}{Heuristics}

\newcommand{\Z}{\mbox{\bf Z}}
\newcommand{\F}{\mbox{\bf  F}}

\renewcommand{\P}{\mathcal{P}}
\newcommand{\Proj}{\mathbf{P}}
\renewcommand{\L}{\mathcal{L}}
\renewcommand{\H}{\mathcal{H}}
\renewcommand{\v}{\mathbf{v}}
\renewcommand{\r}{\mathbf{r}}

\newcommand{\ceil}[1]{\lceil #1 \rceil}

\begin{document}

\title{Traps to the BGJT-Algorithm for Discrete Logarithms}

\author{Qi Cheng\\
School of Computer Science\\
University of Oklahoma\\
Norman, OK\\
\texttt{qcheng@cs.ou.edu}\\
\and
Daqing Wan\\
Department of Mathematics\\
University of California\\
Irvine, CA\\
\texttt{dwan@math.uci.edu}\\
\and Jincheng Zhuang\\
University of Oklahoma\\
Norman, OK\\
\texttt{jzhuang@ou.edu}\\
}
\date{}
\maketitle

\begin{abstract}
In the recent breakthrough paper by Barbulescu, 
Gaudry, Joux and Thom{\'e}, a quasi-polynomial time 
algorithm (QPA) is proposed for the discrete logarithm problem over finite fields
of small characteristic. The time complexity analysis of the algorithm is 
based on several heuristics presented in their  paper. 
We show that some of the heuristics
are problematic  in their original forms, 
in particular, when the field is not a Kummer extension. 
We believe that the basic idea behind the new approach should still work,
and propose a fix to the algorithm in non-Kummer cases,
without altering the quasi-polynomial time complexity. 
The modified algorithm is also heuristic. 
Further study is required in order
to fully understand the effectiveness of the new approach.
\end{abstract}

\section{Introduction}

Many cryptography protocols rely on hard computational 
number theoretical problems for security. 
The discrete logarithm problem over finite fields is one of the most important 
candidates, besides the integer factorization problem.
The hardness of discrete logarithms underpins 
the security of the widely adopted 
Diffie-Hellman key exchange protocol \cite{DiffieHe76} and 
ElGamal's cryptosystem \cite{ElGamal85}.

The  state-of-the-art general-purpose methods for 
solving the discrete logarithm problem in finite fields 
are the number field sieve and 
the function field sieve, which originated from 
the index-calculus algorithm. All the algorithms run
in subexponential time. Let 
\[L_{N}(\alpha)=\exp(O((\log N)^\alpha (\log\log N)^{1-\alpha})).\]
For a finite field $\F_q$,
successful efforts have been made to reduce the heuristic 
complexity of these algorithms 
from $L_{q} (1/2)$ to $L_{q} (1/3)$.
See \cite{Pollard78, Adleman79, Merkle79, Coppersmith84, Gordon93, Adleman94a, JouxLe06a, JouxLe06}.

A sequence of breakthrough results  \cite{Joux13a, Joux13b, GGMZ13} 
recently on the discrete logarithm problem over finite fields culminated 
in a discovery of a quasi-polynomial algorithm 
for small characteristic fields \cite{BGJT13}.
For a finite field $ \F_{q^{2 k}} $ with $ k < q $,
their algorithm runs in heuristic time  $ q^{O(\log k)} $. 
This result, if correct, essentially removes the discrete logarithm
over small characteristic  fields from hard problems in cryptography.

\subsection{Where does the computation really happen? }

Most serious attacks on the discrete logarithm problem over finite fields 
are based on smoothness of integers or  polynomials.
A polynomial is $ m $-smooth if all its irreducible factors have  degrees 
$\leq m  $.
The probability that a random polynomial of degree $n$ ($ \geq m  $ ) 
over a finite field 
$\F_q$ is $m-$smooth is  about $ (n/m)^{-n/m} $ 
\cite{PGF98}.

Suppose that we need to compute discrete logarithm in the field 
$ \F_{q^{2k}} $ where $q > k>1$. 
A main technique in \cite{BGJT13}, which bases on smooth polynomials,
is to find a nice ring generator
$ \zeta $ of $ \F_{q^{2k}}=\F_{q^2}[\zeta]$ over $  \F_{q^2} $ 
satisfying \[ x^q = h_0(x)/h_1(x),  \] 
where $ h_1 $ and $ h_0 $ are polynomials of very small degree.
In many places of the computation, polynomial degrees can be dropped 
quickly by replacing $ x^q $ with $ h_0 (x)/h_1 (x) $, which
allows an effective attack based on smoothness. 

The main issue with this approach is that
the computation really takes place in the ring
$ \F_{q^2}[x]/(x^q h_1 (x) - h_0 (x)) $, where in the analysis of 
\cite{BGJT13},
the computation is assumed to be in $ \F_{q^2}[x]/ (f(x)) $, 
where $f(x)$ is the minimal polynomial of $ \zeta $ over $ \F_{q^2} $.   
Since $f(x)$ divides $   x^q h_1 (x) - h_0 (x)$,  there is a natural surjective ring homomorphism   
\[ \F_{q^2}[x]/(x^q h_1 (x) - h_0 (x)) \rightarrow  \F_{q^2}[x]/ (f(x)). \] 
But  the 
former ring, which is a direct sum of the latter field (if $f(x)$ is a simple factor of 
$x^q h_1 (x) - h_0 (x)$) and a few other  rings,
 is much larger in many cases. 
The computation thus
can be affected by the other rings, rendering
 several conjectures in \cite{Joux13b, BGJT13}  problematic.

\subsection{Our work}

Interestingly, for the Kummer extension of the form
$  \F_{q^2}[x]/ (x^{q-1} - a) $, everything is fine. This is because 
the difference between the ring $ \F_{q^2}[x]/ (x^q - ax) $ 
and the field is rather small.
The discrete logarithm of $x$, which is a  zero divisor in the former ring,
can be computed easily in the latter field, since it  belongs to a subgroup
of a small order ( dividing $(q-1) (q^2-1)$)  in the  field. This is consistent with 
all announced practical  implementations. 

However, in case of more difficult non-Kummer extensions,  we discover that there are multiple problems. 
First, if $x^q h_1(x)-h_0(x)$ has linear factors over $\F_{q^2}$,  the discrete logarithms of
these linear factors cannot be computed in polynomial time,
invalidating a basic assumption in \cite{BGJT13}. One can verify that
most of polynomials given in \cite[Table 1]{Joux13b} have linear factors.
Second, even at the stage of finding discrete logarithms of linear elements, 
we show that there are additional serious restrictions on the choice of $h_0$ and $h_1$. 
For example, if  $x^q h_1(x)-h_0(x)$ has another irreducible factor over $\F_{q^2}$ of degree $k_i$  satisfying 
$\gcd(k_i,k)>1$, we do not see how the algorithm can work. We propose to select $h_0$ and $h_1$ 
such that $x^qh_1(x)-h_0(x)$ has only one irreducible factor $f(x)$ over $\F_{q^2}$ of degree $k$, 
and all other irreducible factors over $\F_{q^2}$ have degrees bigger than one and relatively prime to $k$. Under these assumptions, 
we give an algorithm which will find the discrete logarithm of any linear 
element in polynomial time, under a heuristic assumption supported by our theoretical results 
and numerical data. 

For a non-linear element, a clever idea, the so-called
QPA-descent,  was proposed in
\cite{BGJT13} to reduce its degree, until its relation to linear factors 
can be found. 
While the above two problems about linear factors can be fixed under 
our newly improved heuristic 
assumptions, another serious problem
is  that there are  {\em traps} in the QPA-descent.
For these traps, the QPA-descent described in \cite{BGJT13} 
will not work at all.
They will also block the descent of other elements,
hence severely affecting the usefulness of the new algorithm.
We  propose a descent strategy that avoids the traps,
without altering the quasi-polynomial time complexity. 
The modified algorithm is also heuristic.
We have done a few numerical studies to confirm the 
heuristic. 

In summary, for  large non-Kummer
fields,  we believe that the problem can be significantly 
more subtle than previously thought and 
further study needs to be conducted in order 
to fully understand the effect of the new algorithm.



\section{Finding the discrete logarithm of the linear factors}

We first review the new algorithm in \cite{BGJT13}. 
Suppose that the discrete logarithm
is sought over the field $\F_{q^{2k}}$ with $k < q$.
For other small characteristic fields,
for example, $\F_{p^k} $ ( $ p < k $  ), 
one first embeds it into a slightly larger field:
\[ \F_{p^k} \rightarrow \F_{q^k} \rightarrow  \F_{q^{2k}} 
\]
where $q= p^{\ceil{\log_p k}}$. A quasi-polynomial time
algorithm for  $  \F_{q^{2k}} $ implies a quasi-polynomial
time algorithm for $\F_{p^k} $.
We  assume that
\[ \F_{q^{2k}} = \F_{q^2}[\zeta] \]
where $ \zeta^q = \frac{h_0 (\zeta) }{h_1 (\zeta)}  $.
Here $h_0   $ and $ h_1 $ are polynomials over $ \F_{q^2} $ 
relatively prime to each other, and
of a constant degree. 
In particular, $\deg(h_0) < q +\deg(h_1)$. 
To find such a nice ring generator $\zeta$, one searches over 
all the polynomials
$h_0(x)$ and $h_1(x)$ of a constant degree in $\F_{q^2}[x]$, until
 $h_1(x) x^q - h_0 (x)$ has an irreducible factor $ f(x) $ 
of degree $k$ with multiplicity one.  Let the factorization be
\begin{equation} \label{ringpolfac}
 x^q h_1 (x) - h_0 (x) = f(x) \prod_{i=1}^l  (f_i (x))^{a_i} 
\end{equation}
where the polynomials $f(x)$ and $f_i(x)$'s are irreducible and 
pair-wise prime.
Denote the degree of $f_i (x)$ by $k_i$.

\begin{remark}
In practice, it is enough to search 
only a  quadratic polynomial $ h_0 $ (not necessarily monic)
and a monic linear polynomial $ h_1 $ in $\F_{q^2}[x] $.  
However proving the existence of such polynomials for any 
constant degree such that  $x^q h_1 (x) - h_0 (x)$ has the 
desired factorization pattern
seems to be out of reach by current techniques.
\end{remark}

For simplicity we assume that $ h_1(x) $ is monic and linear.
Most of the known algorithms start by  computing the
discrete logarithms of elements in a special set called a factor base,
which usually contains  small integers, or low degree polynomials.
In the new approach \cite{Joux13b, BGJT13},
the factor base consists of  the linear polynomials 
$ \zeta + \alpha  $ for all $ \alpha \in \F_{q^2} $, and
an algorithm is designed to compute the discrete logarithms of 
all the elements in the factor base. It is conjectured that
this algorithm runs in polynomial time.
One starts the algorithm with the identity:
\[ \prod_{\alpha \in \F_q } (x-\alpha) = x^q - x.  \]
Then apply the Mobius transformation
\[ x \mapsto \frac{a x + b}{cx +d} \]
where the matrix $ m = \begin{pmatrix} a & b \\ c & d
\end{pmatrix} \in \F_{q^2}^{2 \times 2}$ is nonsingular.
We have 
\[ \prod_{\alpha \in \F_q } (\frac{a x + b}{cx +d} -\alpha) 
= (\frac{a x + b}{cx +d})^q - \frac{a x + b}{cx +d}  \]
Clearing the denominator:
\begin{eqnarray*}
&&(cx +d) \prod_{\alpha \in \F_q } ((a x + b) -\alpha (cx +d)) \\
&=& (a x + b)^q (cx +d) - (a x + b) (cx +d)^q \\
&=& (a^q x^q + b^q) (cx+d) - (ax+b)(c^q x^q + d^q).  
\end{eqnarray*}
Multiplying both sides by $h_1(x)$ and replacing $x^qh_1(x)$ by $h_0(x)$, we obtain 
\begin{eqnarray*} 
&& h_1 (x) (cx +d) \prod_{\alpha \in \F_q } ( (a x + b) - \alpha (cx +d))\\
&=&  (a^q {h_0(x)} + b^q h_1 (x) ) (cx+d)
- (ax+b)(c^q {h_0(x)} + d^q h_1(x) )  \\
&& \pmod { x^q h_1(x) - h_0(x)}.
\end{eqnarray*}
If the right-hand side  can be 
factored into a product of linear factors over $ \F_{q^2} $,
we obtain a relation of the form 
\begin{equation}\label{mrelation}
\lambda^{e_0} \prod_{i=1}^{q^2} (x + \alpha_i)^{e_i} =\prod_{i=1}^{q^2} (x + \alpha_i)^{e'_i}\pmod { x^q h_1(x) - h_0(x)} ,
\end{equation}
where $\lambda$ is a multiplicative generator of $\F_{q^2}$,
 $ \alpha_1 = 0, \alpha_2, \alpha_3, \dotsc, \alpha_{q^2}  $
is a natural ordering of elements in $ \F_{q^2} $, and $ e_i $'s
and $ e'_i $'s are non-negative integers.

Following the same notations in \cite{BGJT13}, 
let $\P_q$ be a set of representatives of 
the left cosets of $PGL_2(\F_q)$ in $PGL_2(\F_{q^2})$. 
Note that the cardinality of $\P_q$ is $q^3+q$.
It was shown in \cite{BGJT13} that the matrices in
the same coset produce the same relation (\ref{mrelation}).

Suppose that  for some $1 \leq g \leq q^2 $, 
$ \zeta+\alpha_g $ is a known multiplicative
generator of $ \F_{q^2}[\zeta] = \F_{q^2}[x]/(f(x)) $.  
Since  (\ref{mrelation}) also holds modulo $f(x)$, 
taking the discrete logarithm
w.r.t. the base $ \zeta + \alpha_g $,
we obtain 
\begin{equation}\label{arelation}
e_0  \log_{\zeta+\alpha_g} \lambda + \sum_{1\leq i\leq q^2, i\not= g} 
(e_i-e'_i) \log_{\zeta+\alpha_g} (\zeta+\alpha_i) \equiv e'_g - e_g  
\pmod{q^{2k}-1}.  
\end{equation} 
The above equation gives us a linear   relation among the discrete 
logarithm of linear factors.
One hopes to collect enough relations such that the linear system formed 
by those relations is non-singular over $ \Z/(q^{2k}-1)\Z $.
It allows us to solve $\log_{\zeta+\alpha_g} (\zeta+\alpha_i)  $  
for all the $ \zeta +\alpha_i $ 
in the factor base.

However, 
if for some $ 1\leq z \leq q^2  $, 
\[ (x + \alpha_z) | x^q h_1(x) - h_0 (x),\]
the algorithm will unlikely compute
$ \log_{\zeta+\alpha_g} (\zeta+\alpha_z) $. 
It is because that $ x+\alpha_z $ is zero or nilpotent 
(w.l.o.g. let $ f_1 = x+\alpha_z $) in 
the $\F_{q^2}[x]/( (x+\alpha_z)^{a_1}) $ component of the ring
\[ \F_{q^2}[x]/(x^q h_1 (x) - h_0 (x)) 
=\F_{q^2}[x]/(f(x)) \oplus  \bigoplus_{i=1}^l \F_{q^2}[x]/(f_i(x)^{a_i}).\] 
Hence in   (\ref{mrelation}), if $ e_z > 0 $,
$ e'_z $ is positive as well.    Most likely we will have $ e_z = e'_z $,
so the coefficient for
$ \log_{\zeta +\alpha_g} (\zeta + \alpha_z) $ 
in (\ref{arelation}) will always be $ 0 $.

\begin{remark}
If $ e'_z > e_z \geq 1 $, it is possible to
compute $ \log_{\zeta +\alpha_g} (\zeta + \alpha_z) $.
However, this requires the low degree polynomial in the right
hand side of (\ref{mrelation}) to have the factor $ (x+\alpha_z)^2 $, which
is unlikely.
Our numerical data confirm that it
never happens when $ q $ is sufficiently large.
\end{remark}

To compute the discrete logarithm of $\zeta + \alpha_z$,
we have to use additional relations which hold for the 
field $\F_{q^2}[\zeta]$ but may not hold for the bigger ring 
$ \F_{q^2}[x]/(x^q h_1 (x) - h_0 (x)) $.
The equation \[ (\zeta+\alpha_z )^{q^{2k}-1} = 1  \] is
such an example. 
But this does not help in computing its discrete logarithm 
in the field $\F_{q^2}[\zeta]$, if it is the only relation
involving $ \zeta + \alpha_z $.

In general, it is hard to find useful additional relations 
for $ x+\alpha_z $, since for the algorithm to
 work, it is essential that  we replace $ x^q $ by $h_0(x)/h_1(x)$ 
(not replace $f(x)$ by zero) in the relation generating stage. 
Hence it is not clear that  the discrete logarithm of $\zeta + \alpha_z$ can be
computed in polynomial time, invalidating a conjecture in
\cite{BGJT13}. 
\begin{remark}
An exception is in the case of a Kummer extension,
where the zero divisor $ x $ in the ring has
a small order in the field.
\end{remark}

\section{The tale of two lattices}

To fix the above problem in a non-Kummer case, 
we can either change our factor base
to not include the linear factors of $x^q h_1 (x) - h_0 (x)$,
or we can search for $h_0$ and $h_1$ such that
$x^q h_1 (x) - h_0 (x)$  does not have linear factors.
In the following discussion, we will assume that 
$ x^q h_1 (x) - h_0 (x)  $ has no linear factor for simplicity. That is, 
$$k_i:= \deg(f_i) \geq 2 \ (1\leq i\leq l).$$
In this case, the linear factors $x +\alpha_i$'s are invertible in the ring 
$\F_{q^2} [x]/ (x^q h_1(x) - h_0 (x))$ and equation (\ref{mrelation}) reduces to 
\begin{equation}\label{mmrelation}
\lambda^{e_0} \prod_{i=1}^{q^2} (x + \alpha_i)^{e_i-e'_i} = 1 \pmod { x^q h_1(x) - h_0(x)}. 
\end{equation}
We define two fundamental lattices in $\Z^{q^2+1}$:
\begin{eqnarray*} 
\L_1 &=& \{ (e_0, e_1, \dotsc, e_{q^2}) | \lambda^{e_0} \prod_{i=1}^{q^2} (x + \alpha_i)^{e_i} =1
\pmod{ f (x)} \}, \\
\L_2 &=& \{ (e_0, e_1, \dotsc, e_{q^2}) | \lambda^{e_0} \prod_{i=1}^{q^2} 
(x + \alpha_1)^{e_i} =1
\pmod{ x^q h_1 (x) - h_0 (x) } \}.  
\end{eqnarray*}
It is easy to see that 
$\L_2 \subseteq \L_1$.  
Consider the group homomorphism 
\[ \psi_1: \Z^{q^2 + 1} \rightarrow  (\F_{q^2} [x]/ (f(x)))^*
\]
given by 
\[
(e_0, e_1, \dotsc, e_{q^2}) \mapsto \lambda^{e_0} \prod_{i=1}^{q^2} 
(x + \alpha_i)^{e_i}.
\]
The group homomorphism $\psi_2$ is defined in the same way, except that  
modulo $f(x)$ is replaced by modulo  $(x^q h_1(x) - h_0(x))$ respectively.

\begin{theorem}\label{ThSurjective} 
If $\deg(h_1)\leq 2$, then 
the maps $ \psi_1 $ and $\psi_2$ are  surjective.

\end{theorem}

\begin{proof} It is enough to prove that $\psi_2$ is surjective. If not, the image $H$ of $\psi_2$ 
would be a proper subgroup of $(\F_{q^2} [x]/ (x^q h_1(x) - h_0 (x)))^*$. We can then choose 
a non-trivial character $\chi$ of $(\F_{q^2} [x]/ (x^q h_1(x) - h_0 (x)))^*$ which is trivial on the 
subgroup $H$. Since $\chi$ is trivial on $H$ which contains $\F_{q^2}^*$, we can use the Weil bound as given in  Theorem 2.1 in \cite{Wan97}
and deduce that 
$$1+q^2 =|1 +\sum_{\alpha \in \F_{q^2}}\chi(x +\alpha)| \leq (q+\deg(h_1)-2)\sqrt{q^2} \leq q^2.$$
This is a contradiction. It follows that $\psi_2$ must be surjective. 
\end{proof}

Note that in the application of computing discrete logarithms,
it is important that $ \psi_1 $ is surjective. 
As a corollary, we deduce 

\begin{corollary}  If $\deg(h_1)\leq 2$, then 
\begin{itemize}
\item the group $\Z^{q^2 +1}/ \L_1   $ is isomorphic to  the cyclic group $\Z/(q^{2k}-1)\Z$. 
\item the group $\Z^{q^2 +1}/ \L_2   $ is  isomorphic to
\[ \Z/(q^{2k}-1)\Z \oplus \bigoplus_{i=1}^l \Z/(q^{2k_i} - 1)\Z  \bigoplus (\text{a  finite  $p$-group}). \]
\end{itemize}
\end{corollary}

In particular, the group $\Z^{q^2 +1}/ \L_2   $ is not cyclic when 
$ l \geq 1  $. The relation generation stage only gives lattice vectors in $\L_2$, which 
is far from the $\L_1$ if $l\geq 1$. Thus, we need to add more relations to $\L_2$ in order 
to get close to $\L_1$. 

Since $\lambda^{q^2-1} = 1$, the vector $(q^2-1, 0,\cdots, 0)$ is automatically in $\L_2$.  
Let $\L_2^*$ be the lattice in $\Z^{q^2 + 1}$ 
generated by $\L_2$ and the following $q^2$ vectors 
$$(0, q^{2k}-1, 0, \cdots, 0), \cdots, (0,0, \cdots, 0, q^{2k}-1),$$
corresponding to the relations $(x+\alpha_i)^{q^{2k}-1} =1 $ 
modulo $f(x)$ for $\alpha_i \in \F_{q^2}$. 
It is clear that 
$$\L_2^* = \L_2 + (q^{2k}-1)\Z^{q^2+1}.$$
The next result
gives the group structure for the quotient $\Z^{q^2 +1}/ \L_2^* $.  

\begin{theorem}\label{ThIsomorphism}
For $deg(h_1) \leq 2$, there is a group isomorphism 
$$\Z^{q^2 + 1}/\L_2^*\cong \Z/(q^{2k}-1)\Z \oplus \bigoplus_{1\leq i\leq l} 
\Z/(q^{2\gcd(k,k_i)} - 1)\Z.   $$
\end{theorem}

\begin{proof} Recall that 
\[ \Z^{q^2+1}/\L_2 \cong A \stackrel{\mathrm{def}}{=} \Z/(q^{2k}-1)\Z 
\oplus \bigoplus_{i=1}^l \Z/(q^{2k_i} - 1)\Z  \bigoplus (\text{a  
finite  $p$-group}). \]
It is clear that 
\[ A/(q^{2k}-1)A \cong \Z/(q^{2k}-1)\Z \oplus 
\bigoplus_{1\leq i\leq l}\Z/(q^{2\gcd(k_i,k)} - 1)\Z . \]
The kernel of the surjective composed homomorphism 
$$\Z^{q^2+1}\longrightarrow \Z^{q^2+1}/\L_2 \cong A \longrightarrow A/(q^{2k}-1)A$$
is precisely $\L_2 + (q^{2k}-1)\Z^{q^2+1}=\L_2^*$. The desired isomorphism follows. 

\end{proof}

If $\gcd(k_i,k)>1$ for some $i$, then $\L_2^*$ is still far from $\L_1$.  
We would like $\L_2^*$ to be as close to $\L_1$ as possible in a smooth sense.  
For us, the more interesting case is the following 
\begin{corollary} 
Let $deg(h_1) \leq 2$. If $\gcd(k_i, k)=1$ for all $1\leq i\leq l$, 
we have an isomorphism 
$$\Z^{q^2+1}/\L_2^* \cong \Z/(q^{2k}-1)\Z \oplus 
(\Z/(q^{2} - 1)\Z)^l . $$
\end{corollary}

This corollary shows that under the same assumption, the lattice $\L_2^*$ is a smooth approximation of $\L_1$ 
in the sense that the quotient $\L_1/\L_2^*$ is a direct sum of small order cyclic groups. 

The  algorithm to compute the discrete logarithms 
in the factor base essentially samples vectors from the lattice $\L_2$.
Let $\r_1, \r_2, \dotsc, $ be the vectors in $ \L_2$ 
obtained by the relation-finding algorithm, i.e., 
from the relations in (\ref{mmrelation}). 
Let $ \hat{\L_2} $ be the lattice generated by those vectors.
Let $ \hat{\L_1} $ be the lattice generated by $ \hat{\L_2}  $
and the following $q^2+1$ vectors: 
$$(q^2-1, 0, \cdots, 0), (0, q^{2k}-1, 0, \cdots, 0), \cdots, (0,0, \cdots, 0, q^{2k}-1). $$
Computing 
the Hermite (or Smith) Normal Form of $ \hat{\L_1} $ is equivalent to
solving the linear system $ \hat{\L_2} $ in the ring $ \Z/(q^{2k}-1)\Z $. 
It is in general difficult to find bases for 
the two lattices $\L_1$ and $\L_2$ directly.  
One can think that $ \hat{\L_1} $ and $ \hat{\L_2} $ 
are the approximations of $ \L_1 $ and $ \L_2 $ respectively. These approximations can be computed  
by the polynomial time algorithm.
 Obviously,
\[ \hat{\L_2} 
\begin{array}{c}
\subseteq \L_2 \subseteq\\
\subseteq \hat{\L_1} \subseteq
\end{array}
\L_2^* \subseteq \L_1.  \] 
These inclusions induce surjective group homomorphisms 
\[\Z^{q^2 +1}/ \hat{\L_2}
\begin{array}{c}
\rightarrow \Z^{q^2 +1}/ \L_2 \rightarrow \\
\rightarrow \Z^{q^2 +1}/ \hat{\L_1} \rightarrow 
\end{array}
\Z^{q^2+1}/\L_2^* \rightarrow \Z^{q^2 +1}/ \L_1.
\]
If $\Z^{q^2 + 1} /\hat{\L_2}$ is cyclic, then its quotient $\Z^{q^2 + 1} /\L_2$ 
will be cyclic.  
This is false if $l \geq 1$ as we have seen before.  
Similarly,  $ \Z^{q^2 +1}/ \hat{\L_1} $ is not cyclic as its quotient  $\Z^{q^2 + 1} /\L_2^*$ 
is not cyclic if $l \geq 1$. Hence the conjecture in \cite{HuangNa13} 
also needs modification. It seems reasonable to hope that $\hat{\L_1}$ is a good approximation 
to $\L_2^*$ in the sense that the quotient $\L_2^*/\hat{\L_1}$ is a direct sum of small order 
cyclic groups. In the interesting case when $ \gcd(k,k_i) =1 $ 
for all $ 1\leq i\leq l $, our numerical data suggest the following highly plausible 

\begin{heuristics}\label{HSmith}
Assume that $ x^q h_1(x) - h_0(x) $ does not
have linear factors, and $ \gcd(k,k_i) =1 $ 
for all $ 1\leq i\leq l $. 
Then in the Smith Normal Form of $ \hat{\L_1} $, the diagonal
elements are
\[ 1, 1, \cdots, 1, s_1, \cdots, s_t, q^{2k}-1,  \] 
where for $ 1\leq i\leq t $,  $ s_i > 1 $ and $s_i | q^2 -1$.
\end{heuristics}

Assuming the heuristics, 
$ \Z^{q^2 +1}/ \hat{\L_1} $  is not much bigger than 
$ \Z^{q^2 +1}/ \L_1 $, namely,
\[ \Z^{q^2+1}/\hat{\L_1} \cong \Z/s_1\Z \oplus \Z/s_2\Z \oplus \cdots 
\oplus \Z/s_t \Z \oplus \Z/(q^{2k}-1)\Z. \] 
We can  find a generator for each component,
as a product of linear polynomials from the computation 
of the Smith Normal Form. 
Suppose that for $ 1\leq i\leq t $,
the generator for the $ i $-th  component is
\[ \lambda^{e_{i0}} \prod_{ 1\leq j \leq q^2} (x+\alpha_j)^{e_{ij}}. \] 
Since $ s_i | q^2 -1 $, the above generator belongs to $ \F_{q^2} $
in $ \F_{q^2}[x]/(f(x))$. Assuming that it is $ \lambda^{e'_{i0}}$, 
we have 
\[ \lambda^{e_{i0} - e'_{i0}} \prod_{ 1\leq j \leq q^2} (x+\alpha_j)^{e_{ij}} 
 = 1 \pmod{ f(x) }.  \]  
There are $ t $ such relations. Adding them  to $ \hat{\L_1} $,
we will  finally arrive at the lattice $ \L_1 $.
It allows us to find a  generator for
$ (\F_{q^2}[x]/(f(x)))^* $, and to solve the discrete logarithms 
for the factor base, w.r.t. this generator.

\section{The trap to the QPA-descent}

Now we review the QPA-descent.
Suppose that we need to compute the discrete logarithm
of $ W(\zeta) \in \F_{q^{2k}}[\zeta] $, where $ W $  is
a polynomial over $ \F_{q^2} $ of degree $ w >1 $. 
The QPA-descent, firstly proposed in \cite{BGJT13}, 
is to represent $ W(\zeta) $ 
as a product of elements of smaller degree, e.g. $ \leq w/2 $,
in the field $ \F_{q^2} [x]/ (f(x))  $. To do this,
one again starts with the identity:
\[ \prod_{\alpha \in \F_q } (x-\alpha) = x^q - x.  \]
Then apply the  transformation
\[ x \mapsto \frac{a W(x) + b}{c W(x) +d} \]
where the matrix $ m = \begin{pmatrix} a & b \\ c & d
\end{pmatrix} \in \F_{q^2}^{2 \times 2}$ is nonsingular.
We have 
\[ \prod_{\alpha \in \F_q } (\frac{a W(x) + b}{c W(x) +d} -\alpha) 
= (\frac{a W(x) + b}{c W(x) +d})^q - \frac{a W(x) + b}{c W(x) +d}.  \]
Clearing the denominator:
\begin{eqnarray*}
&&(c W(x) +d) \prod_{\alpha \in \F_q } ((a W(x) + b) -\alpha (c W( x) +d)) \\
&=& (a W(x) + b)^q (c W(x) +d) - (a W(x) + b) (c W(x) +d)^q \\
&=& (a^q \tilde{W}(x^q) + b^q) (cW(x)+d) - 
(a W(x)+b)(c^q \tilde{W}(x^q) + d^q),  
\end{eqnarray*}
where $ \tilde{W} (x) $ is a polynomial obtained by raising the coefficients of 
$ W(x) $ to the $q$-th power.
Replacing $x^q $ with $h_0(x)/h_1(x)$, 
we obtain 
\begin{eqnarray*}
&&  (c W(x) +d) \prod_{\alpha \in \F_q } ( (a W( x) + b) - \alpha (c W(x) +d))\\
&=&  (a^q \tilde{W}(h_0(x)/h_1(x)) + b^q  ) (c W(x)+d) \\
&&- (a W(x) + b ) (\tilde{W}(h_0(x)/h_1(x)) + d^q h_1(x) ) \\
&& \pmod { x^q h_1(x) - h_0(x)}.
\end{eqnarray*}
It was shown in \cite{BGJT13} that  matrices in the 
same left coset of $PGL_2(\F_q)$
of $PGL_2(\F_{q^2})$ generate 
the same equations.
The denominator of the right-hand size  is a power of $ h_1(x) $.  
Denote the numerator of the right-hand side polynomial by $ N_{m,W}(x) $.
If the polynomial  $ N_{m,W}(x) $ is $ w/2 $-smooth, 
namely, it can be factored completely into a product of irreducible factors 
over $ \F_{q^2} $, all have degree $ w/2 $ or less,
we obtain a relation of the form 
\begin{equation}\label{Dmrelation}
\prod_{i=1}^{q^2} (W(x) + \alpha_i)^{e_i} 
 = \lambda^{e_0} \prod_{g(x)\in S } g(x)^{e'_g} \pmod { x^q h_1(x) - h_0(x)} ,
\end{equation}
where $ S\subseteq \F_{q^2}[x] $ is a set of monic polynomials of degrees
less than $ w/2 $ and with cardinality at most $ 3 w $. 
Denote the vector $ (e_1, e_2, \dotsc, e_{q^2}) $ by $ \v_m $.
Note that it is a binary vector, and it is independent of $ W(x) $.  
Collecting enough number of relations will allow
us to represent $ W(x) $ as a product of elements of smaller degrees.
This process is the QPA-descent. A heuristic, made in \cite{BGJT13},
is that repeating the process, one can
represent any element in $ \F_{q^2}[x]/(f(x)) $ 
as a product of linear factors. Combining it with
the fact that the discrete logarithm of the linear factors
are known, one solves the discrete logarithm for any element. 

However the descent will not work
if $ W(x) $ is a factor of $x^q h_1(x) - h_0 (x)  $. Recall that 
$ \alpha_1=0 $.

\begin{theorem}
If 
$ W(x) | x^q h_1(x) - h_0 (x)$,  $e_1 $ will always be $ 0$
in (\ref{Dmrelation}).
\end{theorem}
In other words, if $W(x)$ is a factor of $x^q h_1(x) - h_0(x)  $,
then it will  never appear in the left-hand side of
(\ref{Dmrelation}) as a factor. So the 
descent for $ W(\zeta) $ is not possible.

\begin{proof}
The polynomial
$ W(x) $ is a zero divisor 
in the ring
$ \F_{q^2}[x] / (x^q h_1 (x) - h_0 (x))$.
Hence if     $ W(x) $ 
appears in the left-hand side of (\ref{Dmrelation}) as a factor,
it will also appear in the right-hand side. This contradicts 
the requirement that the factors in the right-hand
side  have degrees smaller than the degree of $ W(x) $. 
\end{proof}

Note that the trap factor $ W(\zeta) $ can appear in the descent paths
of other elements, which essentially blocks the descents.
It is especially troublesome if
 $x^q h_1 (x) - h_0 (x)$ has many small degree factors.


\section{ The trap-avoiding descent  }

Now we have discovered traps for the original QPA-descent. 
How can we work around them? From the above discussion, 
we assume that we work in a non-Kummer extension,
and the polynomial $ x^q h_1 (x) - h_0 (x)  $ 
with the factorization as (\ref{ringpolfac}) satisfies
\begin{itemize}
  \item $ \deg(h_0) \leq 2, \deg(h_1)\leq 1 $;
  \item $ k_i > 1 $ for all $ 1\leq i\leq l $; In other words, 
it is free of linear factors; 
  \item $ \gcd(k, k_i) =1 $  for all $ 1\leq i\leq l $. 
\end{itemize}
In the most interesting case when $ k $ is a prime, our numerical
data show that the above requirements can be easily satisfied.

\begin{heuristics}\label{Hh0h1}
  Let $ q $ be a prime power and $ k < q $ be a prime. 
Then there exist polynomials $ h_0 $ and $ h_1 $ satisfying
the above requirements.
\end{heuristics}

Assume that the  discrete logarithms of all linear polynomials
have been computed. Suppose that we need to compute the discrete logarithm
of $ W(\zeta) $, where  $ W(x) $ is an irreducible polynomial of
degree less than $ k $, and it is relatively prime to $f(x)$. 
If $ W(x) |  x^q h_1(x) - h_0 (x)$,
we will search for an integer $ i $   
such that $ W(x)^i \pmod{ f(x)} $ is relatively prime to $ x^q h_1(x) - 
h_0 (x) $. Such $ i $ can be found easily by a random process.


Now we can  assume that $ \gcd(W(x), x^q h_1(x) - h_0 (x)) =1 $.
If there are not many traps,  we will 
use a trap-avoiding strategy for the descent.
The basic idea is simple.
Whenever we find a relation (\ref{Dmrelation}),
we will not use it unless the right-hand side is
relatively prime to $ x^q h_1(x) - 
h_0 (x) $. 
\begin{definition}
Define the trap-avoiding descent 
lattice $ \L(W) $  associated with $ W(x) $ to be
generated by
\[ \{ \v_m | N_{m,W} \ is \ w/2-smooth, \ and\ 
\gcd(N_{m,W}, x^qh_1 (x)-h_0(x))=1 \}. \]
\end{definition}
Note that  we  use less
relations than \cite{BGJT13} does, since we have to avoid traps. 
If the vector  $(1,0,\ldots,0)$ is in the 
trap-avoiding descent lattice of $ W(x) $,
then $ W(x) $  can be written as a product
of low degree polynomials in  $\F_{q^2}[x]/(f(x))$ that are not traps. 
We believe that the following heuristics is very likely to be true. 
\begin{heuristics}\label{Hdescent}
The trap-avoiding descent lattice for $ W(x) $ 
contains the vector $(1,0,\ldots,0)$ if 
$ \gcd(W(x), x^q h_1(x) - h_0 (x)) =1 $.
\end{heuristics}

To provide a theoretical evidence,  we will show 
that $(1,0,\ldots,0)$ is in its super lattice
that is generated by $ \v_m $ for all $ m\in \P_q $, regardless
whether $ N_{m,W} (x) $  is $w/2$-smooth or not.
This is a slight improvement over \cite{BGJT13},
where it is proved that 
$(q^3-q, 0, \ldots, 0)$ is in the super lattice.
To proceed, we first make some definitions following 
\cite{BGJT13}.
There are two matrices in consideration.
The matrix $\H$ is composed by the binary row vectors $\v_m $ for all 
$  m = \left( \begin{array}{cc} a & b  \\ c & d \end{array} \right) \in \P_q$.
It is a matrix of $q^3+q$ rows and $q^2$ columns.
If we view $ m^{-1} $ as a map from 
$ \Proj^1 (\F_q) $ to $ \Proj^1 (\F_{q^2}) $ 
given by \[ (\beta_1 : \beta_2) \rightarrow ( -d \beta_1 + b \beta_2 : 
c \beta_1 - a \beta_2), \] 
then the $ i $-th component of $ v_m $  is $ 1 $ iff there is a point 
$ P \in \Proj^1 (\F_q) $ such that $ m^{-1}(P) = (\alpha_i : 1) $. 
We define a binary vector $\v_m^{+}=(e_1,\dotsc,e_{q^2},e_{q^2+1})$
for $  m \in \P_q $, where $ (e_1,\dotsc,e_{q^2}) = \v_m, $   
and 
\[ e_{q^2+1}= \left\{ 
 \begin{array}{ll}
  1 & \text{if } (a:c) \in \Proj^1 (\F_q) \\
  0 & \text{otherwise.}
\end{array}
\right.
\]
One can verify that the last component of 
$\v_m^{+}$ corresponds to whether there is a point 
$ P \in \Proj^1 (\F_q) $ such that $ m^{-1}(P) = (1:0) = \infty $.
The matrix $\H^{+}$ is composed by the vectors $\v_m^{+}, m\in \P_q$.
$\H^{+}$ is a matrix of $q^3+q$ rows and $q^2+1$ columns. All the row vectors
have exactly $ q+1 $ many coordinates which are $ 1 $'s.

Denote the lattices generated by the row vectors of $\H$ and $\H^{+}$ by $\L(\H)$ 
and $\L(\H^{+})$ respectively. 
In \cite{BGJT13}, the authors showed that $\v_1=(q^2+q,\ldots,q^2+q)\in \L(\H^{+})$ 
and $\v_2=(q^2+q,q+1,\ldots,q+1)\in \L(\H^{+})$.

\begin{theorem}
The vector $(1, 0, \dotsc, 0) $ is in the lattice $\L(\H)$.
\end{theorem}

\begin{proof} Fix a $ \gamma$ such that   $ \F_{q^2} = \F_q[ \gamma] $. 
Firstly, observe that $\v_3=(1,\ldots,1,q)\in \L(\H^{+})$. 
This follows from $\v_3=\sum_{\beta\in\F_q}\v_{m_\beta}\in \L(\H^{+}),$ 
where $m_\beta=\left( \begin{array}{cc} 1 & \beta\gamma \\ 0 & 1 \end{array} \right) \in \P_q$.
There are $q+1$  row vectors in $\H^{+}$   
such that both the first and the last coordinates are $1$.
Since the projective linear map
on a projective line is sharply 3-transitive, 
a third  coordinate with value $ 1 $  will uniquely determine the coset
in $\P_q $. Thus the sum of these $q+1$ vectors is 
$\v_4=(q+1,1,\ldots,1,q+1)\in \L(\H^{+})$.

From the above observations, we have 
\[
\v_5 = \v_2-(q+1)\v_3= (q^2-1,0,\dotsc,0,1-q^2)\in \L(\H^{+}),
\]
\[
\v_6 = \v_4 - \v_3 = (q,0,\dotsc,0,1)\in \L(\H^{+}).
\]
We deduce 
\[
\v_7 = q\v_6 - \v_5 = (1,0,\dotsc,0,q^2+q-1) \in \L(\H^{+}),
\]
which implies $(1,0,\dotsc,0)\in \L(\H)$. 
\end{proof}

\section{Concluding Remarks and Open problems}

In this paper, we study the validation of the heuristics made 
in the quasi-polynomial time algorithm solving the discrete logarithms
in the small characteristic fields \cite{BGJT13}.  We find that
the heuristics are problematic in the cases of non-Kummer extensions. 
We propose a few modifications to the algorithm, including
some extra requirements for the polynomials $ h_0 $ and $ h_1 $,
and a trap-avoiding descent strategy.  The modified algorithm 
relies on three improved heuristics. 
\begin{proposition}
  If  Heuristics \ref{HSmith}, \ref{Hh0h1} and \ref{Hdescent} hold, then 
the discrete logarithm problem 
over $ \F_{q^k} $ ($ k<q $)  can be solved in time $ q^{O(\log (k))}  $.
\end{proposition}
We believe that proving 
(or disproving ) them 
are interesting open problems that help to understand the effectiveness
of the new algorithm.

\bibliographystyle{plain}
\bibliography{crypto}

\end{document}